\documentclass{amsart}%
\usepackage{amsmath}
\usepackage{graphicx}%
\usepackage{amsfonts}%
\usepackage{amssymb}
\newtheorem{theorem}{Theorem}[section]

\newtheorem{corollary}[theorem]{Corollary}

\newtheorem{definition}[theorem]{Definition}

\newtheorem{lemma}[theorem]{Lemma}

\newtheorem{proposition}[theorem]{Proposition}
\newtheorem{remark}[theorem]{Remark}

\begin{document}
\title[Scaffolds and integral Hopf Galois module structure]{Scaffolds and integral Hopf Galois module structure on purely inseparable extensions}
\author{Alan Koch}

\begin{abstract}
Let $p$ be prime. Let $L/K$ be a finite, totally ramified, purely inseparable
extension of local fields, $\left[  L:K\right]  =p^{n},\;n\geq2.$ It is known
that $L/K$ is Hopf Galois for numerous Hopf algebras $H,$ each of which can
act on the extension in numerous ways. For a certain collection of such $H$ we
construct ``Hopf Galois scaffolds'' which allow us to obtain a Hopf analogue
to the Normal Basis Theorem for $L/K.$ The existence of a scaffold structure
depends on the chosen action of $H$ on $L.$ We apply the theory of scaffolds
to describe when the fractional ideals of $L$ are free over their associated
orders in $H.$

\end{abstract}
\maketitle

\section{Introduction}

Let $L/K$ be a totally ramified extension of local fields of degree $p^{n}$,
where the residue field of $K$ has characteristic $p.$ Suppose further that
$L/K$ is Galois with $G=\operatorname*{Gal}\left(  L/K\right)  .$ Let
$\mathfrak{O}_{K}$ and $\mathfrak{O}_{L}$ denote the valuation rings of $K$
and $L$ respectively. There are two natural ways to describe the elements of
$L$, namely by using its valuation $v_{L}\;$or by using its Galois action. If
$\pi\in L$ is a uniformizing parameter, then every element of $L$ is a
$K$-linear combination of powers of $\pi$; computing its valuation is a simple
process. A drawback of the valuation representation of $L$ is that the Galois
action is not necessarily transparent.

Alternatively, we have the Normal Basis Theorem, which asserts that there
exists a $\rho\in L$ whose Galois conjugates form a $K$-basis for $L/K;$
equivalently, $L$ is a free rank one module over the group algebra $KG.$ Here,
every element of $L$ is a $K$-linear combination of $\left\{  \sigma\left(
\rho\right)  :\sigma\in G\right\}  ,$ which allows for a simple description of
the Galois action; however, the valuation representation is not transparent,
making certain Galois module theory questions difficult to answer. For
example, $\mathfrak{O}_{L}$ is an $\mathfrak{O}_{K}G$-module, however by
Noether's Theorem \cite{Noether31} $\mathcal{\mathfrak{O}}_{L}$ is not free of
rank one if $L/K$ is wildly ramified. The $\mathfrak{O}_{K}G$-module structure
of $\mathfrak{O}_{L}$ when $\mathfrak{O}_{L}$ does not possess a normal
integral basis can be more difficult. A typical strategy, thanks to Leopoldt
\cite{Leopoldt59} is to replace $\mathfrak{O}_{K}G$ with a larger
$\mathfrak{O}_{K}$-subalgebra of $KG,$ namely
\[
\mathfrak{A}=\left\{  \alpha\in KG:\alpha\left(  \mathfrak{O}_{L}\right)
\subset\mathfrak{O}_{L}\right\}  ,
\]
which also acts on $\mathfrak{O}_{L};$ the structure of $\mathfrak{O}_{L}$ as
an $\mathfrak{A}$-module can be simpler to describe.

In an attempt to unite these representations, G. Griffith Elder \cite{Elder09}
first developed a theory of ``Galois scaffolds''. In that work a Galois
scaffold consists of a subset $\left\{  \theta_{1},\theta_{2},\dots,\theta
_{n}\right\}  $ of $KG$, together with a positive integer $v,$ called an
integer certificate, such that $\left\{  v_{L}\left(  \theta_{i}^{j}\left(
\rho\right)  \right)  :1\leq i\leq n,\;0\leq j\leq p-1\right\}  $ is a
complete set of residues $\operatorname{mod}p^{n}$ where $\rho\in L$ is any
element of valuation $v$. Certainly, $\left\{  \theta_{i}^{j}\left(
\rho\right)  \right\}  $ forms a $K$-basis for $L,$ and this basis facilitates
the study of both valuation and Galois action, particularly if $\theta_{i}%
^{p}=0$ for all $i$. A simple example of a Galois scaffold arises when $n=1$
and the break number $b$ is relatively prime to $p;$ in this case, if
$G=\left\langle \sigma\right\rangle $ then $\theta_{1}=\sigma-1,\;v=b$ is an
example of a Galois scaffold. Such scaffolds do not always exist -- in fact,
integer certificates may not exist, for example if $L/K$ is unramified and
$\pi^{p}=1$ \cite{Byott11}. This notion of scaffold was refined in
\cite{ByottElder13}, and then again in \cite{ByottChildsElder14},\ the latter
version being the most useful for describing the integral Galois module structure.

The version in \cite{ByottChildsElder14} is also the most general as it does
not insist that $L/K$ be Galois, merely that there is a $K$-algebra $A$ which
acts on $L$ in a very reasonable way. A classic example of such an algebra is
a $K$-Hopf algebra. There are many more Hopf Galois extensions than Galois
extensions. For example, any Galois extension is Hopf Galois for at least one
Hopf algebra (namely, $H=KG$) and, if $n\geq2,$ many more: the exact
determination of the number of such $H$ is a group theory problem thanks to
\cite{GreitherPareigis87}, which covers all separable extensions. At the other
extreme, if the extension $L/K$ is purely inseparable, then it is also Hopf
Galois \cite{Chase76}; if $\left[  L:K\right]  \geq p^{2},$ then there are
numerous Hopf algebras which make $L/K$ Hopf Galois \cite{Koch14a}.

In the setting where $L/K$ is Hopf Galois with Hopf algebra $H$, one can study
the structure of $\mathfrak{O}_{L}$ as an $H$-module. Given
\cite{ByottChildsElder14}, a natural approach would be an attempt to construct
an $H$-scaffold which, loosely, consists of $\left\{  \lambda_{t}%
:t\in\mathbb{Z}\right\}  \subset L$ with $v_{L}\left(  \lambda_{t}\right)
=t,$ along with $\left\{  \Psi_{i}:0\leq i\leq n-1\right\}  \subset H$ such
that $\Psi_{i}$ acts on $\lambda_{t}$ in a manner which makes $v_{L}\left(
\Psi_{i}\left(  \lambda_{t}\right)  \right)  $ easy to compute.

Here, we focus on the case where $L/K$ is a totally ramified, purely
inseparable extension of local fields, $\left[  L:K\right]  =p^{n},\;n\geq2.$
We take a collection of Hopf algebras $H$ which make $L/K$ Hopf Galois and
describe the generalized integral Hopf Galois module structure of
$\mathfrak{O}_{L}.$ The integral Hopf Galois module structure we seek is a
description of all of the fractional ideals of $L$ as $H$-modules. In detail,
each fractional ideal of $L$ is of the form $\mathfrak{P}_{L}^{h}$ for
$h\in\mathbb{Z},$ where $\mathfrak{P}_{L}$ is the maximal ideal of
$\mathfrak{O}_{L}.$ In other words, $\mathfrak{P}_{L}^{h}=\left\{  x\in
L:v_{L}\left(  x\right)  \geq h\right\}  .$ For each $h$ we let $\mathfrak{A}%
_{h}$ be the largest subset of $H$ which acts on $\mathfrak{P}_{L}^{h},$ i.e.,%
\[
\mathfrak{A}_{h}=\left\{  \alpha\in H:\alpha\mathfrak{P}_{L}^{h}%
\subset\mathfrak{P}_{L}^{h}\right\}  .
\]
We call $\mathfrak{A}_{h}$ the associated order of $\mathfrak{P}_{L}^{h}$ in
$H$: it is clearly an $\mathfrak{O}_{K}$-subalgebra of $H$ and $\mathfrak{A}%
_{h}\otimes_{\mathfrak{O}_{K}}K\cong H.$ By construction, $\mathfrak{A}_{h}$
acts on $\mathfrak{P}_{L}^{h}$; the existence of the scaffold allows for a
numerical criterion for determining whether $\mathfrak{P}_{L}^{h}$ is a free
$\mathfrak{A}_{h}$-module. The criterion itself is independent of the
scaffold, provided the scaffold exists.

The paper is organized as follows. After giving a definition of an
$H$-scaffold, a simpler version than the one in \cite{ByottChildsElder14}, we
consider the family of monogenic $K$-Hopf algebras $H_{n,r,f},\;1\leq r\leq
n-1,\;f\in K^{\times}$ introduced in \cite{Koch14a} which make $L$ an
$H_{n,r,f}$-Galois object. We examine the case where $2r\geq n$ and consider
actions of the linear dual $H:=H_{n,r,f}^{\ast}$ which give $L/K$ the
structure of a Hopf Galois extension. A subtlety that arises is that $H$
possesses an infinite number of actions on $L$; in each case, $L/K$ is
$H$-Galois. The different actions correspond with different choices for
$K$-algebra generator $x\in L$; and for each choice of $x$ we will construct
$H$-scaffolds for infinitely many actions. As with the Galois case, the
$H$-scaffold will allow us to consider the effect of the action on the
valuation of some specially chosen elements, and using \cite[Th 3.1,
3.7]{ByottChildsElder14} we will use it to describe the integral Hopf Galois
module structure. We will then focus on a specific action for which an
$H$-scaffold exists, and explicitly describe which fractional ideals
$\mathfrak{P}_{L}^{h}$ are free over their associated orders. We conclude with
some remarks concerning selecting the ``best'' choices of $r$ and $f,$ and the
action on $L,$ for answering integral Hopf Galois module theory questions.

The evident purpose of this work is to construct $H$-scaffolds. However, our
results contribute to the bigger picture of scaffolds. The definition of a
scaffold has evolved significantly since Elder's 2009 paper, which required
$L/K$ to be a Galois extension. At this point, it is not yet clear how
prevalent scaffolds are for general Hopf Galois extensions. But we will see
that in the finite purely inseparable case, many scaffolds exist.

Throughout, we fix an integer $n\geq2$ and $L$ a totally ramified purely
inseparable extension of $K=\mathbb{F}_{q}\left(  \left(  T\right)  \right)  $
of degree $p^{n}.$ Let $v_{K}$ be the $T$-adic valuation, $v_{L}$ the
extension of $v_{K}$ to $L$. Write $L=K\left(  x\right)  ,\;x^{p^{n}}=\beta\in
K,\;v_{K}\left(  \beta\right)  =-b<0,\;p\nmid b.$ We let $H$ and $H_{n,r,f}$
be as above, and we assume $2r\geq n.$

The author would like to thank G.\ Griffith Elder for his input in the
preparation of this paper, and the University of Nebraska at Omaha for their
generous hospitality during the development of some of these results.

\section{Scaffolds}

The definition of an $A$-scaffold in \cite{ByottChildsElder14} is very general
-- more so than we need here. We will simplify this definition as much as
possible, and since our acting $K$-algebra is a Hopf algebra we will refer to
it as an $H$-scaffold.

\begin{definition}
\label{bcedef}Let $a$ be an integer such that $ab\equiv-1\,\operatorname{mod}%
p^{n}.$ Let $\mathfrak{T}>1$ be an integer. An $H$-scaffold on $L$ of
tolerance $\mathfrak{T}$ consists of:

\begin{enumerate}
\item A set $\left\{  \lambda_{j}:j\in\mathbb{Z},\;v_{L}\left(  \lambda
_{j}\right)  =j\right\}  $ of elements of $L$ such that $\lambda_{j_{1}%
}\lambda_{j_{2}}^{-1}\in K$ when $j_{1}\equiv j_{2}\,\operatorname{mod}p^{n}.$

\item A collection $\left\{  \Psi_{s}:0\leq s\leq n-1\right\}  $ of elements
in $H$ such that $\Psi_{s}\left(  1_{K}\right)  =0$ for all $s$ and, mod
$\lambda_{j+p^{s}b}\mathfrak{P}_{L}^{\mathfrak{T}},$%
\[
\Psi_{s}\left(  \lambda_{j}\right)  \equiv\left\{
\begin{array}
[c]{cc}%
u_{s,j}\lambda_{j+p^{s}b} & \operatorname*{res}\left(  aj\right)  _{s}>0\\
0 & \text{otherwise}%
\end{array}
\right.
\]
where $u_{s,j}\in\mathfrak{O}_{K}^{\times},$ $\operatorname*{res}\left(
aj\right)  $ is the least nonnegative residue of $aj$ mod $p^{n},$ and
\[
\operatorname*{res}\left(  aj\right)  =\sum_{s=0}^{n-1}\operatorname*{res}%
\left(  aj\right)  _{s}\,p^{s},\;0\leq\operatorname*{res}\left(  aj\right)
_{s}\leq p-1
\]
is the $p$-adic expansion of $\operatorname*{res}\left(  aj\right)  .$
\end{enumerate}
\end{definition}

Given an $H$-scaffold we know the effect of applying $\Psi_{s}$ to
$\lambda_{j},$ provided $\operatorname*{res}\left(  aj\right)  _{s}>0.$ For
$0<i\leq p-1$ it can be readily seen that $\operatorname*{res}\left(  a\left(
b+p^{s}bi\right)  \right)  _{s}=p-i>0,$ hence $\Psi_{s}^{i}\left(  \lambda
_{b}\right)  \equiv u\lambda_{b+p^{s}bi}\,\operatorname{mod}\lambda_{j+p^{s}%
b}\mathfrak{P}_{L}^{\mathfrak{T}}$ for some $u\in\mathfrak{O}_{K}^{\times}.$
More generally,
\[
v_{L}\left(  \Psi_{0}^{i_{0}}\Psi_{1}^{i_{1}}\cdots\Psi_{n-1}^{i_{n-1}}\left(
\lambda_{b}\right)  \right)  =b+b\sum_{s=0}^{n-1}i_{s}p^{s},\;0\leq i_{s}\leq
p-1.
\]
By allowing the $\left\{  i_{s}\right\}  $ to vary, we obtain $p^{n}$ elements
of $L,$ pairwise incongruent modulo $p^{n},$ hence $\left\{  \Psi_{0}^{i_{0}%
}\Psi_{1}^{i_{1}}\cdots\Psi_{n-1}^{i_{n-1}}\left(  \lambda_{b}\right)  :0\leq
i_{s}\leq p-1\right\}  \;$is a $K$-basis for $L.$

We will use the result below to construct our $H$-scaffolds.

\begin{lemma}
\label{Hscaf}Suppose we have $\left\{  \Psi_{s}:0\leq s\leq n-1\right\}
\subset H$ such that, for\ $i\leq p^{n}-1,$ $i=\sum_{s=0}^{n-1}i_{s}p^{s},$
\[
\Psi_{s}\left(  x^{i}\right)  \equiv i_{s}x^{i-p^{s}}\,\operatorname{mod}%
x^{i-p^{s}}\mathfrak{P}_{L}^{\mathfrak{T}}%
\]
for some $\mathfrak{T}>1.$ Let%
\[
\lambda_{j}=T^{\left(  j+b\operatorname*{res}\left(  aj\right)  \right)
/p^{n}}x^{\operatorname*{res}\left(  aj\right)  }.
\]
Then $\left\{  \lambda_{j}\right\}  ,\left\{  \Psi_{s}\right\}  $ form a
scaffold of tolerance $\mathfrak{T}.$
\end{lemma}

\begin{proof}
First, since $v_{L}\left(  x\right)  =-b,$%
\[
v_{L}\left(  \lambda_{j}\right)  =j+b\operatorname*{res}\left(  aj\right)
-b\operatorname*{res}\left(  aj\right)  =j,
\]
and clearly $v_{L}\left(  \lambda_{j_{1}}\lambda_{j_{2}}^{-1}\right)
=j_{1}-j_{2},$ so condition (1)\textbf{ }of the definition above is satisfied.
Next, we have%
\begin{align*}
\Psi_{s}\left(  \lambda_{j}\right)   &  =\Psi_{s}\left(  T^{\left(
j+b\operatorname*{res}\left(  aj\right)  \right)  /p^{n}}%
x^{\operatorname*{res}\left(  aj\right)  }\right) \\
&  =T^{\left(  j+b\operatorname*{res}\left(  aj\right)  \right)  /p^{n}}%
\Psi_{s}\left(  x^{\operatorname*{res}\left(  aj\right)  }\right) \\
&  \equiv T^{\left(  j+b\operatorname*{res}\left(  aj\right)  \right)  /p^{n}%
}\operatorname*{res}\left(  aj\right)  _{s}x^{\operatorname*{res}\left(
aj\right)  -p^{s}}\,\operatorname{mod}x^{\operatorname*{res}\left(  aj\right)
-p^{s}}\mathfrak{P}_{L}^{\mathfrak{T}}%
\end{align*}
If $\operatorname*{res}\left(  aj\right)  _{s}=0$ then $\Psi_{s}\left(
\lambda_{j}\right)  =0.$ Otherwise, $a\left(  j+bp^{s}\right)  \equiv
aj-p^{s}\,\operatorname{mod}p^{n}$ and%
\begin{align*}
\operatorname*{res}\left(  a\left(  j+bp^{s}\right)  \right)   &
=\operatorname*{res}\left(  aj-p^{s}\right) \\
&  =\operatorname*{res}\left(  aj\right)  -p^{s},
\end{align*}
the latter equality since $\operatorname*{res}\left(  aj\right)  \geq p^{s}.$
Thus $\operatorname*{res}\left(  aj\right)  =p^{s}+\operatorname*{res}\left(
a\left(  j+bp^{s}\right)  \right)  $ and so%
\begin{align*}
j+b\operatorname*{res}\left(  aj\right)   &  =j+b\left(  p^{s}%
+\operatorname*{res}\left(  a\left(  j+bp^{s}\right)  \right)  \right) \\
&  =j+b\operatorname*{res}\left(  a\left(  j+bp^{s}\right)  \right)  +bp^{s},
\end{align*}
giving%
\begin{align*}
\operatorname*{res}\left(  aj\right)  _{s}T^{\left(  j+b\operatorname*{res}%
\left(  aj\right)  \right)  /p^{n}}x^{\operatorname*{res}\left(  aj\right)
-p^{s}}  &  =\operatorname*{res}\left(  aj\right)  _{s}T^{\left(
j+b\operatorname*{res}\left(  a\left(  j+bp^{s}\right)  \right)
+bp^{s}\right)  /p^{n}}x^{\operatorname*{res}\left(  a\left(  j+bp^{s}\right)
\right)  }\\
&  =\operatorname*{res}\left(  aj\right)  _{s}\lambda_{j+bp^{s}}.
\end{align*}
Setting $u_{s,j}=\operatorname*{res}\left(  aj\right)  _{s}$ shows that (2) is
also satisfied.
\end{proof}

\begin{remark}
By adjusting each $\lambda_{j}$ by a scalar it is possible to have
$u_{s,j}=1.$ This is the primary difference between the construction above and
the one found in \cite[Sec. 5.3]{ByottChildsElder14}.
\end{remark}

In the work to follow, we will use the definition of $H$-scaffold given by the
description in Lemma \ref{Hscaf}. As the choice of $\left\{  \lambda
_{j}\right\}  $ will remain fixed (assuming a constant $b$), we will refer to
the scaffold as $\left\{  \Psi_{s}\right\}  .$

\section{The Hopf Algebra Structure}

In this section, we introduce the class of Hopf algebras we will use to
construct our $H$-scaffolds. To do so, we first recall a family of Hopf
algebras introduced in \cite{Koch14a}. For $0<r<n\leq2r$ and $f\in K^{\times
},$ let $H_{n,r,f}$ be the $K$-Hopf algebra whose $K$-algebra structure is
$H_{n,r,f}=K\left[  t\right]  /\left(  t^{p^{n}}\right)  $; whose counit and
antipodal map are $\varepsilon\left(  t\right)  =0$ and $\lambda\left(
t\right)  =-t$ respectively; and whose comultiplication is
\[
\Delta\left(  t\right)  =t\otimes1+1\otimes t+f\sum_{\ell=1}^{p-1}%
\frac{1}{\ell!\left(  p-\ell\right)  !}t^{p^{r}\ell}\otimes t^{p^{r}\left(
p-\ell\right)  }.
\]

Let us fix values for $r,n,$ and $f$ as above$,$ and let $H=H_{n,r,f}^{\ast}.$
Certainly, $H$ has a $K$-basis $\left\{  z_{0}=1,z_{1},\dots,z_{p^{n}%
-1}\right\}  $ with $z_{i}:H\rightarrow K$ given by
\[
z_{j}\left(  t^{i}\right)  =\delta_{i,j},
\]
where $\delta_{i,j}$ is the Kronecker delta. The algebra structure on $H$ is
induced from the coalgebra structure on $H_{n,r,f};$ explicitly,%
\begin{equation}
z_{j_{1}}z_{j_{2}}\left(  h\right)  =\operatorname{mult}\left(  z_{j_{1}%
}\otimes z_{j_{2}}\right)  \Delta\left(  h\right)  .\label{zmult}%
\end{equation}
In this section we will show that $\left\{  z_{p^{s}}:0\leq s\leq n-1\right\}
$ generate $H$ as a $K$-algebra. This set will be (part of) the scaffolds we develop.

We start by recalling a result which will facilitate the study of the algebra
structure of $H$ as well as the action of $H$ on $L.$

\begin{lemma}
\label{thepowlem}Let%
\[
S_{f}\left(  u,v\right)  =u+v+f\sum_{\ell=1}^{p-1}\frac{1}{\ell!\left(
p-\ell\right)  !}u^{p^{r}\ell}v^{p^{r}\left(  p-\ell\right)  }.
\]
Then, for every positive integer $i$, $S_{f}\left(  u,v\right)  ^{i}$ is an
$K^{\times}$-linear combination of elements of the form
\[
f^{i_{3}}u^{i_{1}+p^{r}\ell^{\prime}}v^{i_{2}+p^{r}\ell^{\prime\prime}},
\]
where%
\begin{align*}
i  &  =i_{1}+i_{2}+i_{3}\\
\ell^{\prime}  &  =i_{3,1}+2i_{3,2}+\cdots+\left(  p-1\right)  i_{3,p-1}\\
\ell^{\prime\prime}  &  =\left(  p-1\right)  i_{3,1}+\left(  p-2\right)
i_{3,2}+\cdots+i_{3,p-1},
\end{align*}
and $i_{3,1}+i_{3,2}+\cdots+i_{3,p-1}=i_{3}.$
\end{lemma}

\begin{proof}
This is a straightforward calculation from \cite[Lemma 5.1]{Koch14a} -- we
recall it here for the reader's convenience.

We have%
\begin{align*}
S_{f}\left(  u,v\right)  ^{i}  &  =\left(  u+v+f\sum_{\ell=1}^{p-1}%
\frac{1}{\ell!\left(  p-\ell\right)  !}u^{p^{r}\ell}v^{p^{r}\left(
p-\ell\right)  }\right)  ^{i}\\
&  =\sum_{i_{1}+i_{2}+i_{3}=i}\binom{i}{i_{1},i_{2},i_{3}}\left(  u^{i_{1}%
}v^{i_{2}}\right)  \left(  f\sum_{\ell=1}^{p-1}\frac{1}{\ell!\left(
p-\ell\right)  !}u^{p^{r}\ell}v^{p^{r}\left(  p-\ell\right)  }\right)
^{i_{3}}.
\end{align*}
The last factor in each summand can be expanded as%
\[
f^{i_{3}}\sum_{i_{3,1}+\cdots+i_{3,p-1}=i_{3}}\left(  \binom{i_{3}}%
{i_{3,1},\dots,i_{3,p-1}}\left(  \prod_{j=1}^{p-1}\frac{1}{i_{3,j}!\left(
p-i_{3,j}\right)  !}\right)  u^{^{i_{1}+p^{r}\ell^{\prime}}}v^{^{i_{2}%
+p^{r}\ell^{\prime\prime}}}\right)  .
\]

The result follows.
\end{proof}

Next, we consider powers of the $z_{p^{s}}$'s.

\begin{lemma}
For $0\leq s\leq r,\;1\leq m\leq p-1;$ or $0\leq s\leq r-1,\;m=p$ we have
$z_{p^{s}}^{m}=m!z_{mp^{s}}.$ In particular, $z_{p^{s}}^{p}=0.$
\end{lemma}

\begin{proof}
See \cite[Lemmas 5.2, 5.3]{Koch14a}. While the result there was for $n=r+1,$
its validity depended on the form of the comultiplication; the more general
$2r\geq n$ case the comultiplication has the same form., and hence a nearly
identical proof.
\end{proof}

The result above does not hold for $s>r.$ However, we do have

\begin{lemma}
\label{prpow}For $0\leq s\leq n-1,$ $1\leq j,m\leq p-1,\ $we have $z_{p^{s}%
}^{j}\left(  t^{mp^{s}}\right)  =m!\delta_{j,m}.$ Furthermore, if $s\geq r$
then $z_{p^{s}}^{p}\left(  t^{p^{i}}\right)  =f^{p^{s-r}}\delta_{i,s-r}.$ In
particular, $z_{p^{s}}^{p}\neq0.$
\end{lemma}

\begin{proof}
Certainly, if $s<r$ then the result follows from the previous lemma. Thus, we
will assume that $s\geq r.$ The statement $z_{p^{s}}^{j}\left(  t^{mp^{s}%
}\right)  =m!\delta_{j,m}$ is clearly true for $j=1.$ Suppose $z_{p^{s}}%
^{j-1}\left(  t^{\left(  m-1\right)  p^{s}}\right)  =\left(  m-1\right)
!\delta_{j,m-1}.$ Since $s+r\geq n$ we have that $t^{p^{s}}$ is a primitive
element, hence%
\begin{align*}
z_{p^{s}}^{j}\left(  t^{mp^{s}}\right)   &  =\operatorname{mult}\left(
z_{p^{s}}^{j-1}\otimes z_{p^{s}}\right)  \left(  t^{p^{s}}\otimes1+1\otimes
t^{p^{s}}\right)  ^{m}\\
&  =\sum_{i=0}^{m}\binom{m}{i}z_{p^{s}}^{j-1}\left(  t^{ip^{s}}\right)
z_{p^{s}}\left(  t^{\left(  m-i\right)  p^{s}}\right)  .
\end{align*}
Recalling that $z_{i}\left(  t^{j}\right)  =0$ for $i\neq j,$ for this to be
nonzero, we require $i=j-1$ and $m-i=1.$ Thus, $m=j$ and
\begin{align*}
z_{p^{s}}^{m}\left(  t^{mp^{s}}\right)   &  =\binom{m}{m-1}z_{p^{s}}%
^{m-1}\left(  t^{p^{s}\left(  m-1\right)  }\right)  z_{p^{s}}\left(  t^{p^{s}%
}\right) \\
&  =m\left(  m-1\right)  !\\
&  =m!,
\end{align*}
proving the first statement of the lemma.

For the second, we have%
\begin{align*}
z_{p^{s}}^{p}\left(  t^{p^{i}}\right)   &  =\operatorname{mult}\left(
z_{p^{s}}^{p-1}\otimes z_{p^{s}}\right)  \left(  t\otimes1+1\otimes
t+f\sum_{\ell=1}^{p-1}\frac{1}{\ell!\left(  p-\ell\right)  !}t^{p^{r}\ell
}\otimes t^{p^{r}\left(  p-\ell\right)  }\right)  ^{p^{i}}\\
&  =z_{p^{s}}^{p-1}\left(  t^{p^{i}}\right)  z_{p^{s}}\left(  1\right)
+z_{p^{s}}^{p-1}\left(  1\right)  z_{p^{s}}\left(  t^{p^{i}}\right)
+f^{p^{i}}\sum_{\ell=1}^{p-1}\frac{1}{\ell!\left(  p-\ell\right)  !}z_{p^{s}%
}^{p-1}\left(  t^{p^{r+i}\ell}\right)  z_{p^{s}}\left(  t^{p^{r+i}\left(
p-\ell\right)  }\right)  .
\end{align*}
Since $z_{p^{s}}\left(  1\right)  =0$ we may ignore the first two terms, and
so%
\[
z_{p^{s}}^{p}\left(  t^{p^{i}}\right)  =f^{p^{i}}\sum_{\ell=1}^{p-1}%
\frac{1}{\ell!\left(  p-\ell\right)  !}z_{p^{s}}^{p-1}\left(  t^{p^{r+i}\ell
}\right)  z_{p^{s}}\left(  t^{p^{r+i}\left(  p-\ell\right)  }\right)  .
\]
In order that a summand be nonzero we require $p^{r+i}\left(  p-\ell\right)
=p^{s},$ i.e. $\ell=p-1,$ and hence $i=s-r.$ We have, since $z_{p^{s}}%
^{p-1}\left(  t^{mp^{s}}\right)  =\left(  p-1\right)  !\delta_{p-1,m},$%
\begin{align*}
z_{p^{s}}^{p}\left(  t^{p^{i}}\right)   &  =f^{p^{s-r}}\frac{1}{\left(
p-1\right)  !}z_{p^{s}}^{p-1}\left(  t^{p^{s}\left(  p-1\right)  }\right) \\
&  =f^{p^{s-r}}\frac{1}{\left(  p-1\right)  !}\left(  p-1\right)  !\\
&  =f^{p^{s-r}}.
\end{align*}
For $i\neq s-r$ we have $z_{p^{s}}^{p}\left(  t^{p^{i}}\right)  =0.$
\end{proof}

It can be shown that the set
\[
\left\{  \prod_{s=0}^{n-1}z_{p^{s}}^{j_{s}}:0\leq j_{s}\leq p-1\right\}
\]
is a $K$-basis for $H.$ A formal proof will be given in section 5. By counting
dimensions, it is clear that $z_{p^{s}}^{p^{2}}=0$ for $r\leq s\leq n-1.$

The coalgebra structure on $H$ is induced from the multiplication on
$H_{n,r,f}$ and is simply
\[
\Delta\left(  z_{j}\right)  =\sum_{i=0}^{j}z_{j-i}\otimes z_{i}.
\]

\section{The Hopf Galois Action}

In \cite{Koch14a} we describe how $L$ can be viewed as an $H_{n,r,f}$-Galois
object. Since $2r\geq n$ the $K$-algebra map $\alpha:L\rightarrow L\otimes
H_{n,r,f}$ given by%
\begin{equation}
\alpha\left(  x\right)  =x\otimes1+1\otimes t+f\sum_{\ell=1}^{p-1}%
\frac{1}{\ell!\left(  p-\ell\right)  !}x^{p^{r}\ell}\otimes t^{p^{r}\left(
p-\ell\right)  }\label{thisf}%
\end{equation}
provides an $H_{n,r,f}$-comodule structure on $L$; furthermore, the map
$\gamma:L\otimes L\rightarrow L\otimes H_{n,r,f}$ given by $\gamma\left(
x^{i}\otimes x^{j}\right)  =x^{i}\alpha\left(  x^{j}\right)  $ is an
isomorphism, hence $L$ is an $H_{n,r,f}$-Galois object. In this section, we
describe the induced action of $H=H_{n,r,f}^{\ast}$ on $L$ which makes $L/K$
an $H$-Galois extension.

Before proceeding, notice that this action depends on two choices:\ the choice
of $x,$ the $K$-algebra generator for $L,$ and the choice of $t,$ the
$K$-algebra generator for $H_{n,r,f}.$ By replacing $x$ with $x^{\prime
},\;p\nmid v_{L}\left(  x^{\prime}\right)  $ we may define
\[
\alpha^{x^{\prime}}\left(  x^{\prime}\right)  =x^{\prime}\otimes1+1\otimes
t+f\sum_{\ell=1}^{p-1}\frac{1}{\ell!\left(  p-\ell\right)  !}\left(
x^{\prime}\right)  ^{p^{r}\ell}\otimes t^{p^{r}\left(  p-\ell\right)  }%
\]
and obtain a different coalgebra structure. Alternatively, if we replace $t$
with, say, $t_{g}:=gt,\;g\in K^{\times}$ we may define
\[
\alpha_{g}\left(  x\right)  =x\otimes1+1\otimes t_{g}+fg^{1-p^{r+1}}\sum
_{\ell=1}^{p-1}\frac{1}{\ell!\left(  p-\ell\right)  !}x^{p^{r}\ell}\otimes
t_{g}^{p^{r}\left(  p-\ell\right)  }%
\]
which also results in a different coalgebra structure. Furthermore, each of
the coalgebra structures here give $L$ the structure of an $H_{n,r,f}$-Galois
object. Combined, we have coactions given by%
\[
\alpha_{h}^{y}\left(  y\right)  =y\otimes1+1\otimes t_{h}+fh^{1-p^{r+1}}%
\sum_{\ell=1}^{p-1}\frac{1}{\ell!\left(  p-\ell\right)  !}y^{p^{r}\ell}\otimes
t_{h}^{p^{r}\left(  p-\ell\right)  },\;y\in L^{\times},\;p\nmid v_{L}\left(
y\right)  ,\;h\in K^{\times}%
\]
although some choices of $h,y$ produce the same actions, e.g. $\alpha_{1}%
^{x}=\alpha_{T^{-1}}^{Tx}.$ By fixing $x\in L$ we eliminate some of the
ambiguity as to which coaction is being used. For the rest, notice that
$K\left[  t_{g}\right]  /\left(  t_{g}^{p^{n}}\right)  =H_{n,r,fg^{1-p^{r+1}}%
}$, and so $H_{n,r,f}=H_{n,r,fg^{1-p^{r+1}}}$ for any choice of $g\in
K^{\times}$, hence choosing the $K$-algebra generator for the Hopf algebra is
equivalent to choosing a representative of a coset in $K^{\times}/\left(
K^{\times}\right)  ^{p^{r+1}-1};$ once such a choice $f$ is made, it is
assumed that the coaction of $H_{n,r,f}$ follows the coaction given in eq.
$\left(  \ref{thisf}\right)  $. In other words, we will always use the action
$\alpha_{1}^{x}.$

Generally, if $A$ is a $K$-Hopf algebra such that $L$ is an $A$-Galois object,
then $A^{\ast}$ acts on $L$ by%
\begin{equation}
h\left(  y\right)  =\operatorname{mult}\left(  1\otimes h\right)
\alpha\left(  y\right)  ,\;h\in A^{\ast},y\in L. \label{act}%
\end{equation}
As $H$ is generated by $\left\{  z_{p^{s}}:0\leq s\leq n-1\right\}  ,$ it
suffices to compute $z_{p^{s}}\left(  x^{i}\right)  $ for $0\leq s\leq
n-1,\;1\leq i\leq p^{n}-1$.

\begin{proposition}
\label{action}For $0\leq i\leq p^{n}-1$, write%
\[
i=\sum_{s=0}^{n-1}i_{s}p^{\ell}.
\]
Then, for $0\leq s\leq r-1$ we have%
\[
z_{p^{s}}\left(  x^{i}\right)  =i_{s}x^{i-p^{s}}.
\]
Additionally,%
\[
z_{p^{r}}\left(  x^{i}\right)  =i_{r}x^{i-p^{r}}-ifx^{p^{r}\left(  p-1\right)
+i-1}.
\]
\end{proposition}

\begin{remark}
Note that if $i<p^{s}$ then $z_{p^{s}}\left(  x^{i}\right)  =0,$ and if
$i<p^{r}$ then $z_{p^{r}}\left(  x^{i}\right)  =-ifx^{p^{r}\left(  p-1\right)
+i-1}.$
\end{remark}

\begin{proof}
We have%
\begin{align*}
z_{p^{s}}\left(  x^{i}\right)   &  =\operatorname{mult}\left(  1\otimes
z_{p^{s}}\right)  \alpha\left(  x^{i}\right) \\
&  =\operatorname{mult}\left(  1\otimes z_{p^{s}}\right)  S_{f}\left(
x\otimes1,1\otimes t\right)  ^{i}\\
&  =\operatorname{mult}\left(  1\otimes z_{p^{s}}\right)  \left(
x\otimes1+1\otimes t+f\sum_{\ell=1}^{p-1}\frac{1}{\ell!\left(  p-\ell\right)
!}x^{p^{r}\ell}\otimes t^{p^{r}\left(  p-\ell\right)  }\right)  ^{i}\\
&  =\operatorname{mult}\left(  1\otimes z_{p^{s}}\right)  \sum_{i_{1}%
+i_{2}+i_{3}=i}\binom{i}{i_{1},i_{2},i_{3}}\left(  x^{i_{1}}\otimes t^{i_{2}%
}\right)  \left(  f\sum_{\ell=1}^{p-1}\frac{1}{\ell!\left(  p-\ell\right)
!}x^{p^{r}\ell}\otimes t^{p^{r}\left(  p-\ell\right)  }\right)  ^{i_{3}}.
\end{align*}
When simplified, the tensors are of the form $x^{i_{1}+i_{3}p^{r}\ell^{\prime
}}\otimes t^{i_{2}+i_{3}p^{r}\ell^{\prime\prime}},$ $\ell^{\prime}%
,\ell^{\prime\prime}$ as before. Applying $1\otimes z_{p^{s}}$ to each tensor
will give $0$ unless%
\begin{equation}
p^{s}=i_{2}+i_{3}p^{r}\ell^{\prime\prime}. \label{spow}%
\end{equation}
Assume first that $s<r.$ Since $p^{r}>p^{s}$ we see that $\ell^{\prime\prime
}=0.$ This can only occur if $i_{3}=0.$ Thus $i_{2}=p^{s}$ and $i_{1}%
=i-p^{s},$ giving%
\begin{align*}
z_{p^{s}}\left(  x^{i}\right)   &  =\binom{i}{i-p^{s},p^{s},0}x^{i-p^{s}%
}z_{p^{s}}\left(  t^{p^{s}}\right) \\
&  =\binom{i}{p^{s}}x^{i-p^{s}}\\
&  =i_{s}x^{i-p^{s}},
\end{align*}
the last equality following from Lucas' Theorem (see \cite{Fine47}). Thus
$z_{p^{s}}\left(  x^{i}\right)  =i_{s}x^{i-p^{s}},$ as desired.

Now we consider the case $s=r.$ Then $i_{3}=0,\;i_{2}=p^{r},\;i_{1}=i-p^{r}$
certainly satisfies eq. $\left(  \ref{spow}\right)  $. However, we get an
additional solution to this equation, namely $i_{3}=1,\;\ell^{\prime
}=p-1,\;\ell^{\prime\prime}=1,\;i_{2}=0,\;i_{1}=i-1$ -- as
\[
i_{2}+i_{3}p^{r}\ell^{\prime\prime}=p^{r}\left(  p-\ell\right)  ,
\]
with this solution we have the left-hand side equal $0+p^{r}\left(  1\right)
=p^{r},$ hence $\ell=p-1$. Thus%
\begin{align*}
z_{p^{r}}\left(  x^{i}\right)   &  =\binom{i}{i-p^{r},p^{r},0}x^{i-p^{r}%
}z_{p^{r}}\left(  t^{p^{r}}\right)  +\binom{i}{i-1,0,1}x^{i-1}f\frac{1}%
{\left(  p-1\right)  !\left(  p-\left(  p-1\right)  \right)  !}x^{p^{r}\left(
p-1\right)  }z_{p^{r}}\left(  t^{p^{r}}\right) \\
&  =i_{r}x^{i-p^{r}}-ifx^{p^{r}\left(  p-1\right)  +i-1}.
\end{align*}
\end{proof}

Much like it was for the algebra structure, describing the action for $s>r$ is
more complicated as eq. $\left(  \ref{spow}\right)  $ can have numerous
solutions. However, in the sequel we will be able to effectively study how the
valuation of an element of $L$ changes when $z_{p^{s}}$ is applied.

\section{A Scaffold on $H$}

Recall that $L=K\left(  x\right)  ,\;x^{p^{n}}=\beta,\;v_{L}\left(  x\right)
=v_{K}\left(  \beta\right)  =-b,\;p\nmid b.$ In this section we build an
$H$-scaffold for $L$ using the action above. Initially, we will insist on a
restriction on $f,$ however this restriction will ultimately not be necessary.

We start by determining the effect of applying $z_{p^{s}}$ to powers of $x.$
The first result is fundamental.

\begin{proposition}
\label{tol}Let $0\leq s\leq n-1,\;1\leq i\leq p^{n}-1.$ Write $i=\sum
_{s=0}^{n-1}i_{s}p^{s}.\;If$ $v_{K}\left(  f\right)  \geq bp^{r+1-n}$ then
\[
z_{p^{s}}\left(  x^{i}\right)  \equiv i_{s}x^{i-p^{s}}\,\operatorname{mod}%
x^{i-p^{s}}\mathfrak{P}_{L}^{\mathfrak{T}}%
\]
where $\mathfrak{T}=p^{n}v_{K}\left(  f\right)  -b\left(  p^{r+1}-1\right)  .$
\end{proposition}

\begin{proof}
Since $z_{p^{s}}\left(  x^{i}\right)  =\left(  1\otimes z_{p^{s}}\right)
\left(  \alpha\left(  x\right)  \right)  ^{i}$ we have%
\[
z_{p^{s}}\left(  x^{i}\right)  =\sum_{i_{1}+i_{2}+i_{3}=i}f^{i_{3}}%
\sum_{i_{3,1}+\cdots+i_{3,p-1}=i_{3}}c_{i_{1},i_{2},i_{3}}x^{i_{1}+p^{r}%
\ell^{\prime}}z_{p^{s}}\left(  t^{i_{2}+p^{r}\ell^{\prime\prime}}\right)
\]
where $c_{i_{1},i_{2},i_{3}}\in K^{\times}$ and, as before,%
\begin{align*}
\ell^{\prime}  &  =i_{3,1}+2i_{3,2}+\cdots+\left(  p-1\right)  i_{3,p-1}\\
\ell^{\prime\prime}  &  =\left(  p-1\right)  i_{3,1}+\left(  p-2\right)
i_{3,2}\cdots+i_{3,p-1}.
\end{align*}
For a summand to be nontrivial we require $i_{2}+p^{r}\ell^{\prime\prime
}=p^{s},$ in which case the summand is a $K^{\times}$-multiple of $f^{i_{3}%
}x^{i_{1}+p^{r}\ell^{\prime}}.$

If $s<r$ then Lemma \ref{action} gives
\[
z_{p^{s}}\left(  x^{i}\right)  =i_{s}x^{i-p^{s}},
\]
and clearly the desired congruence holds.

Now suppose $s\geq r.$ Then $z_{p^{s}}\left(  x^{i}\right)  $ will again
contain the summand $i_{\left(  s\right)  }x^{i-p^{s}}$ arising from
$i_{3}=0,$ however there may be positive choices of $i_{3}$ which make
$i_{2}+p^{r}\ell^{\prime\prime}=p^{s}.$ Since $i_{3}\leq\ell^{\prime\prime
}\leq\left(  p-1\right)  i_{3}$ it follows that $i_{3}\leq p^{s-r}.$ For an
$\ell^{\prime\prime}$ in this interval we have $i_{2}=p^{s}-p^{r}\ell
^{\prime\prime}$ and $i_{1}=i-\left(  p^{s}-p^{r}\ell^{\prime\prime}\right)
-i_{3}.$ Since $\ell^{\prime}+\ell^{\prime\prime}=pi_{3},$ the $i_{3}>0$ terms
in the summand are all of the form%
\begin{align*}
c_{i_{1},i_{2},i_{3}}f^{i_{3}}x^{i-\left(  p^{s}-p^{r}\ell^{\prime\prime
}\right)  -i_{3}+p^{r}\left(  pi_{3}-\ell^{\prime\prime}\right)  }  &
=c_{i_{1},i_{2},i_{3}}f^{i_{3}}x^{i-p^{s}+i_{3}\left(  p^{r+1}-1\right)  }\\
&  =\left(  c_{i_{1},i_{2},i_{3}}f^{i_{3}}x^{i_{3}\left(  p^{r+1}-1\right)
}\right)  x^{i-p^{s}}.
\end{align*}
Thus%
\[
z_{p^{s}}\left(  x^{i}\right)  =i_{s}x^{i-p^{s}}+\sum\left(  c_{i_{1}%
,i_{2},i_{3}}f^{i_{3}}x^{i_{3}\left(  p^{r+1}-1\right)  }\right)  x^{i-p^{s}%
},
\]
where the sum is taken over all $i_{1},i_{2},i_{3}$ with $i_{3}>0.$ Now for
$i_{3}\geq1,$%
\begin{align*}
v_{L}\left(  f^{i_{3}}x^{i_{3}\left(  p^{r+1}-1\right)  }\right)   &
=p^{n}i_{3}v_{K}\left(  f\right)  -b\left(  i_{3}\left(  p^{r+1}-1\right)
\right) \\
&  =i_{3}\left(  p^{n}v_{K}\left(  f\right)  -bp^{r+1}+b\right)  ,
\end{align*}
and since $v_{K}\left(  f\right)  \geq bp^{r+1-n}$ this expression is
minimized when $i_{3}$ is minimized, i.e., $i_{3}=1.$ Thus%
\[
v_{L}\left(  f^{i_{3}}x^{i_{3}\left(  p^{r+1}-1\right)  }\right)  \geq
p^{n}v_{K}\left(  f\right)  -b\left(  p^{r+1}-1\right)  ,
\]
so $f^{i_{3}}x^{i_{3}\left(  p^{r+1}-1\right)  }\in\mathfrak{P}_{L}%
^{\mathfrak{T}}$, $\mathfrak{T}=p^{n}v_{K}\left(  f\right)  -b\left(
p^{r+1}-1\right)  .$ Hence,
\[
z_{p^{s}}\left(  x^{i}\right)  =i_{s}x^{i-p^{s}}\left(  1+\sum c_{i_{1}%
,i_{2},i_{3}}f^{i_{3}}x^{i_{3}\left(  p^{r+1}-1\right)  }\right)
\]
and so%
\[
z_{p^{s}}\left(  x^{i}\right)  \equiv i_{s}x^{i-p^{s}}\,\operatorname{mod}%
x^{i-p^{s}}\mathfrak{P}_{L}^{\mathfrak{T}}.
\]
\end{proof}

As we have seen, the restriction on $v_{K}\left(  f\right)  $ is not a
restriction on the Hopf algebra, merely on the ways in which this Hopf algebra
can act on $L.$ We must write $H=H_{n,r,f}^{\ast},\;v_{K}\left(  f\right)
\geq bp^{r+1-n}$ for the action (induced from the coaction in eq. $\left(
\ref{thisf}\right)  $ for this choice of $f$) to provide an $H$-scaffold. As
$H_{n,r,f}=H_{n,r,T^{p^{p+1}-1}f},$ it is clear that there will be an infinite
number of actions of $H$ on $L$ which produce scaffolds. To ensure a scaffold
of tolerance $\mathfrak{T}>1$ we require a slight increase in the lower bound
for $v_{K}\left(  f\right)  $. For the rest of the section, we shall assume
$v_{K}\left(  f\right)  >bp^{r+1-n}.$

\begin{theorem}
\label{scaffy}For $v_{K}\left(  f\right)  >bp^{r+1-n},$ the set $\left\{
z_{1}^{j_{0}}z_{p}^{j_{1}}\cdots z_{p^{n-1}}^{j_{n-1}}:0\leq j_{s}\leq
p-1\right\}  $ constructed above is an $H$-scaffold on $L$ with tolerance
$\mathfrak{T}=p^{n}v_{K}\left(  f\right)  -b\left(  p^{r+1}-1\right)  >1.$
\end{theorem}

The presentation of the scaffold above follows the form given in Lemma
\ref{Hscaf}. To obtain a scaffold which follows Definition \ref{bcedef}, we
pick an integer $a$ such that $ab\equiv-1$ $\left(  \operatorname{mod}%
p^{n}\right)  $ and set%
\[
\lambda_{j}=T^{\left(  j+b\operatorname*{res}\left(  aj\right)  \right)
/p^{n}}x^{\operatorname*{res}\left(  aj\right)  },\;j\in\mathbb{Z}.
\]
This set, together, with $\left\{  \Psi_{s}=z_{p^{s}}:0\leq s\leq n-1\right\}
$, forms the scaffold on $L$ of tolerance $\mathfrak{T}$ as in the sense of
Definition \ref{bcedef}. In particular,
\begin{align}
\lambda_{b}  &  =T^{\left(  b+b\operatorname*{res}\left(  ab\right)  \right)
/p^{n}}x^{\operatorname*{res}\left(  ab\right)  }\label{scaffold}\\
&  =T^{\left(  b+b\left(  p^{n}-1\right)  \right)  /p^{n}}x^{p^{n}%
-1}\nonumber\\
&  =T^{b}x^{p^{n-1}}.\nonumber
\end{align}

As an immediate consequence, we get:

\begin{corollary}
\label{KbasisL}The set
\[
\left\{  \prod_{s=0}^{n-1}z_{p^{s}}^{j_{s}}\left(  \lambda_{b}\right)  :0\leq
j_{s}\leq p-1\right\}
\]
is a $K$-basis for $L.$
\end{corollary}

\begin{proof}
This follows from the discussion between Definition \ref{bcedef} and Lemma
\ref{Hscaf}. In particular, note that%
\[
\left\{  v_{L}\left(  \prod_{s=0}^{n-1}z_{p^{s}}^{j_{s}}\left(  \lambda
_{b}\right)  \right)  :0\leq s\leq n-1,0\leq j_{s}\leq p-1\right\}
\]
forms a complete set of residues mod $p^{n}.$
\end{proof}

We devote the remainder of this section to showing that the action of $H$ on
$L$ has an ``integer certificate''. In classical Galois module theory, a
number $c\in\mathbb{Z}$ is called an integer certificate if, for all $\rho\in
L$ with $v_{L}\left(  \rho\right)  =c,$ the set $\left\{  \sigma\left(
\rho\right)  :\sigma\in\operatorname*{Gal}\left(  L/K\right)  \right\}  $ is a
$K$-basis for $L.$ We modify that here: a number $c\in\mathbb{Z}$ is an
integer certificate if whenever $v_{L}\left(  \rho\right)  =c$ the set%
\[
\left\{  z_{1}^{j_{0}}z_{p}^{j_{1}}\cdots z_{p^{n-1}}^{j_{n-1}}\left(
\rho\right)  :0\leq j_{s}\leq p-1\right\}
\]
is a $K$-basis for $L.$

As an immediate consequence to Proposition \ref{tol} we get

\begin{corollary}
Let $0\leq s\leq n-1,\;1\leq i\leq p^{n}-1.$ Suppose $z_{p^{s}}\left(
x^{i}\right)  \neq0.$ Then $v_{L}\left(  z_{p^{s}}\left(  x^{i}\right)
\right)  =b\left(  p^{s}-i\right)  =v_{L}\left(  x^{i}\right)  +bp^{s}.$
\end{corollary}

As each application of $z_{p^{s}}$ increases valuation by $bp^{s}$, the above
result allows us to determine the effect, on valuation, of applying our basis
elements of $H$ to the standard $K$-basis of $L.$

\begin{corollary}
\label{valcor}Let $1\leq i\leq p^{n}-1,$ and let $0\leq j_{s}\leq p-1$ for all
$0\leq s\leq n-1.$ If $z_{1}^{j_{0}}z_{p}^{j_{1}}\dots z_{p^{n-1}}^{j_{n-1}%
}\left(  x^{i}\right)  \neq0$ then
\[
v_{L}\left(  z_{1}^{j_{0}}z_{p}^{j_{1}}\dots z_{p^{n-1}}^{j_{n-1}}\left(
x^{i}\right)  \right)  =v_{L}\left(  x^{i}\right)  +b\sum_{s=0}^{n-1}%
j_{s}p^{s}.
\]
\end{corollary}

To set some notation, given $0\leq j\leq p^{n}-1,$ we define $0\leq
j_{0},\dots,j_{n-1}\leq p-1$ to be the unique integers such that%
\[
j=\sum_{s=0}^{n-1}j_{s}p^{s}.
\]
Conversely, given a collection $\left\{  j_{0},\dots,j_{n-1}\right\}  $ with
$0\leq j_{s}\leq p-1$ for all $0\leq s\leq n-1$ we define $j$ using the
summation above.

We claim that if $v_{L}\left(  \rho\right)  =b$ then
\[
\left\{  z_{1}^{j_{0}}z_{p^{2}}^{j_{1}}\dots z_{p^{n-1}}^{j_{n-1}}\left(
\rho\right)  :0\leq j_{\ell}\leq p-1\right\}
\]
forms a basis for $L/K.$ The crucial step to establishing this is the following.

\begin{proposition}
\label{valprop}Pick $\rho\in L$ with $v_{L}\left(  \rho\right)  =b.\;$Then
\[
v_{L}\left(  z_{1}^{j_{0}}z_{p}^{j_{1}}\dots z_{p^{n-1}}^{j_{n-1}}\left(
\rho\right)  \right)  =b\left(  1+j\right)  .
\]
\end{proposition}

\begin{proof}
Any $\rho\in L$ with $v_{L}\left(  \rho\right)  =b$ has the form
\[
\rho=g\left(  x^{-1}+\sum_{\ell=1}^{p^{n}}a_{\ell}x^{-1-\ell}\right)
\]
with $g\in K,\;a_{\ell}\in K,$ $v_{K}\left(  g\right)  =0,\ $and $v_{L}\left(
a_{\ell}\right)  >-b\ell$ for all $1\leq\ell\leq p^{n}.$ Let us write
$g=g_{0}T^{b}x^{p^{n}},$ and for simplicity we assume $g_{0}=1.$ Then%
\[
\rho=T^{b}x^{p^{n}-1}+T^{b}\sum_{\ell=1}^{p^{n}}a_{\ell}x^{p^{n}-1-\ell}%
\]
(note that $T^{b}x^{p^{n}-1}$ is the element $\lambda_{b}$ from eq.
\ref{scaffold}, and thus is part of the scaffold in the Definition
\ref{bcedef} sense) and%
\[
z_{1}^{j_{0}}z_{p}^{j_{1}}\dots z_{p^{n-1}}^{j_{n-1}}\left(  \rho\right)
=T^{b}z_{1}^{j_{0}}z_{p}^{j_{1}}\dots z_{p^{n-1}}^{j_{n-1}}\left(  x^{p^{n}%
-1}\right)  +T^{b}\sum_{\ell=1}^{p^{n}}a_{\ell}z_{1}^{j_{0}}z_{p}^{j_{1}}\dots
z_{p^{n-1}}^{j_{n-1}}\left(  x^{p^{n}-1-\ell}\right)  .
\]
Applying Corollary \ref{valcor} to the case where $i=p^{n}-1-\ell,$ either
$z_{1}^{j_{0}}z_{p}^{j_{1}}\dots z_{p^{n-1}}^{j_{n-1}}\left(  x^{p^{n}-1-\ell
}\right)  =0$ or
\[
v_{L}\left(  z_{1}^{j_{0}}z_{p}^{j_{1}}\dots z_{p^{n-1}}^{j_{n-1}}\left(
x^{p^{n}-1-\ell}\right)  \right)  =-b\left(  p^{n}-1-\ell\right)  +bj.
\]
Furthermore, observe that%
\[
z_{1}^{j_{0}}z_{p^{2}}^{j_{1}}\dots z_{p^{n-1}}^{j_{n-1}}\left(  x^{p^{n}%
-1}\right)  \neq0,\;0\leq j_{\ell}\leq p-1
\]
since $p^{n}-1=\left(  p-1\right)  +\left(  p-1\right)  p+\cdots+\left(
p-1\right)  p^{n-1}.$ Thus,
\begin{align*}
v_{L}\left(  T^{b}z_{1}^{j_{0}}z_{p}^{j_{1}}\dots z_{p^{n-1}}^{j_{n-1}}\left(
x^{p^{n}-1}\right)  \right)   &  =bp^{n}-b\left(  p^{n}-1\right)  +bj\\
&  =b\left(  1+j\right)
\end{align*}
since
\[
v_{L}\left(  T^{b}a_{\ell}z_{1}^{j_{0}}z_{p}^{j_{1}}\dots z_{p^{n-1}}%
^{j_{n-1}}\left(  x^{p^{n}-1-\ell}\right)  \right)  \geq p^{n}b+v_{L}\left(
a_{\ell}\right)  -b\left(  p^{n}-1-\ell\right)  +bj
\]
and, since $v_{L}\left(  a_{\ell}\right)  >-b\ell,$%
\begin{align*}
p^{n}b+v_{L}\left(  a_{\ell}\right)  -b\left(  p^{n}-1-\ell\right)  +bj  &
=p^{n}b+v_{L}\left(  a_{\ell}\right)  -bp^{n}+b+b\ell+bj\\
&  =v_{L}\left(  a_{\ell}\right)  +b\ell+b\left(  1+j\right) \\
&  >b\left(  1+j\right)  =v_{L}\left(  T^{b}z_{1}^{j_{0}}z_{p}^{j_{1}}\dots
z_{p^{n-1}}^{j_{n-1}}\left(  x^{p^{n}-1}\right)  \right)  ,
\end{align*}
hence
\[
v_{L}\left(  z_{1}^{j_{0}}z_{p}^{j_{1}}\dots z_{p^{n-1}}^{j_{n-1}}\left(
\rho\right)  \right)  =\min\left\{  b\left(  1+j\right)  ,v_{L}\left(
a_{\ell}\right)  +b\ell+b\left(  1+j\right)  \right\}  =b\left(  1+j\right)
\]
since the minimum is uniquely achieved.
\end{proof}

\begin{remark}
Generally, it is not the case that if $z_{p^{s}}\left(  y\right)  \neq0$ then
$v_{L}\left(  z_{p^{s}}\left(  y\right)  \right)  =v_{L}\left(  y\right)
+bp^{s},$ i.e., that an application of $z_{p^{s}}$ universally increases
valuation by $bp^{s}.$For example, $v_{L}\left(  x^{p-1}+Tx^{p}\right)
=-\left(  p-1\right)  b$ but $v_{L}\left(  z_{p}\left(  x^{p-1}+Tx^{p}\right)
\right)  =v_{L}\left(  T\right)  =p^{n}.$ However, it is always true that
$v_{L}\left(  z_{p^{s}}\left(  y\right)  \right)  \geq v_{L}\left(  y\right)
+bp^{s}.$
\end{remark}

\begin{corollary}
The set $\left\{  z_{1}^{j_{0}}z_{p}^{j_{1}}\cdots z_{p^{n-1}}^{j_{n-1}%
}\left(  \rho\right)  :0\leq j_{s}\leq p-1\right\}  $ forms a $K$-basis for
$L,$ i.e., $b$ is an integer certificate.
\end{corollary}

\begin{proof}
Observe that%
\[
\left\{  v_{L}\left(  z_{1}^{j_{0}}z_{p^{2}}^{j_{1}}\dots z_{p^{n-1}}%
^{j_{n-1}}\left(  \rho\right)  \right)  :0\leq j_{s}\leq p-1\right\}
=\left\{  b\left(  1+j\right)  :0\leq j\leq p^{n}-1\right\}  .
\]
Now $\left\{  b\left(  1+j\right)  :0\leq j\leq p^{n}-1\right\}  \;$ is a
complete set of residues $\operatorname{mod}p^{n}$ since $p\nmid b.$ Thus,
$\left\{  z_{1}^{j_{0}}z_{p}^{j_{1}}\dots z_{p^{n-1}}^{j_{n-1}}\left(
\rho\right)  :0\leq j_{s}\leq p-1\right\}  $ is $K$-linearly independent, and
hence a basis for $L.$
\end{proof}

\section{Integral Hopf Galois Module Structure}

In this section we describe the Hopf Galois module structure of $\mathfrak{O}%
_{L}$ and of all of the fractional ideals $\mathfrak{P}_{L}^{h}$ of $L.$ Given
a high enough tolerance level, the results of \cite{ByottChildsElder14} enable
us to describe the $H$-module structure of $\mathfrak{P}_{L}^{h}$ \ We apply
their work below, and then we will take a look at a specific action of $H$ on
$L$.

Let $h\in\mathbb{Z}.$ Since $\mathfrak{P}_{L}^{h+p^{n}}=T\mathfrak{P}_{L}^{h}$
and $\mathfrak{A}_{h+p^{n}}=\mathfrak{A}_{h}$ it suffices to consider the Hopf
Galois module structure on a complete set of residues mod $p^{n}.$ We will
pick the set of residues $h$ such that $0\leq b-h\leq p^{n}-1.$

We start with:

\begin{lemma}
There exists actions of $H$ on $L$ which produce $H$-scaffold structures on
$L/K$ with arbitrarily high tolerance.
\end{lemma}

\begin{proof}
As we can write $H=H_{n,r,f}^{\ast}$ with $v_{K}\left(  f\right)  $ of
arbitrarily high valuation, this is clear since $\mathfrak{T}=p^{n}%
v_{K}\left(  f\right)  -b\left(  p^{r+1}-1\right)  $ for $v_{K}\left(
f\right)  \geq bp^{r+1-n}.$
\end{proof}

For the remainder of this section, pick $f$ such that
\[
v_{K}\left(  f\right)  \geq\frac{2p^{n}-1+b\left(  p^{r+1}-1\right)  }{p^{n}},
\]
so $\mathfrak{T}\geq2p^{n}-1.$ This level of tolerance allows us to determine
integral Hopf Galois module structure. 

\begin{remark}
This new bound on $v_{K}\left(  f\right)  $ is larger than the one we imposed
in section 5. While we could have simply assumed $v_{K}\left(  f\right)
\geq\left(  2p^{n}-1+b\left(  p^{r+1}-1\right)  \right)  p^{-n}$ throughout,
we wanted to also provide examples of $H$-scaffolds for which Hopf Galois
module structure could not be completely determined.
\end{remark}

We will now introduce numerical data from \cite{ByottChildsElder14}. For each
$0\leq j\leq p^{n}-1,$ let%
\begin{align*}
d_{h}\left(  j\right)   &  =\left\lfloor \frac{bj+b-h}{p^{n}}\right\rfloor \\
w_{h}\left(  j\right)   &  =\min\left\{  d_{h}\left(  i+j\right)
-d_{h}\left(  i\right)  :0\leq i\leq p^{n}-1,\;i_{s}+j_{s}\leq p-1\text{ for
all }s\right\}  ,
\end{align*}
using our convention that $j=\sum j_{s}p^{s},\;i=\sum i_{s}p^{s}$ as before.
Then, using Theorem 3.1, Theorem 3.7, and Corollary 3.2 of
\cite{ByottChildsElder14} we get all of the following.

\begin{proposition}
\label{bigbce}With the notation as above:

\begin{enumerate}
\item $\mathfrak{A}_{h}$ has $\mathfrak{O}_{K}$-basis $\left\{  T^{-w_{h}%
\left(  j\right)  }z_{1}^{j_{0}}z_{p}^{j_{1}}\cdots z_{p^{n-1}}^{j_{n-1}%
}:0\leq j\leq p^{n}-1\right\}  .$

\item $\mathfrak{O}_{K}$ is a free $\mathfrak{A}$-module of rank one --
explicitly, $\mathfrak{O}_{L}=\mathfrak{A}\cdot\rho,\;v_{L}\left(
\rho\right)  =b$ -- if $\operatorname*{res}\left(  b\right)  \mid\left(
p^{m}-1\right)  $ for some $1\leq m\leq n.$

\item $\mathfrak{P}_{L}^{h}$ is a free $\mathfrak{A}_{h}$-module if and only
if $w_{h}\left(  j\right)  =d_{h}\left(  j\right)  $ for all $0\leq j\leq
p^{n}-1;$ furthermore if this equality holds then $\mathfrak{P}_{L}%
^{h}=\mathfrak{A}_{h}\cdot\rho,\;v_{L}\left(  \rho\right)  =b.$

\item If $w_{h}\left(  j\right)  \neq d_{h}\left(  j\right)  ,$ then
$\mathfrak{P}_{L}^{h}$ can be generated over $\mathfrak{A}_{h}$ using $\ell$
generators, where $\ell=\,^{\#}\left\{  i:d_{h}\left(  i\right)  >d_{h}\left(
i-j\right)  +w_{h}\left(  j\right)  \text{ for all }0\leq j\leq p^{n-1}%
\,\;\text{with\ }j_{s}\leq i_{s}\right\}  $
\end{enumerate}
\end{proposition}

\begin{remark}
It is important to note that the determination as to whether $\mathfrak{P}%
_{L}^{h}$ is free over $\mathfrak{A}_{h}$ does not depend on the $H$-scaffold
itself, merely on the behavior of $d_{h}$ and $w_{h}.$
\end{remark}

\begin{remark}
Note that if $\operatorname*{res}\left(  b\right)  \mid\left(  p^{m}-1\right)
$ then $\mathfrak{O}_{K}$ is free over $\mathfrak{A},$ but in general the
converse does not hold. But since \textbf{2} is a special case of \textbf{3}
where $h=0$ we do have necessary and sufficient conditions for when
$\mathfrak{O}_{K}$ is free over $\mathfrak{A.}$
\end{remark}

Let us interpret these results in the case where $b=1,$ which requires that
$v_{K}\left(  f\right)  \geq3.$ (Note that scaffolds exist for $v_{K}\left(
f\right)  =2,$ as well as for $v_{K}\left(  f\right)  =1$ unless $n=r+1$.)
Then $2-p^{n}\leq h\leq1$ and%
\[
d_{h}\left(  j\right)  =\left\lfloor \frac{j+1-h}{p^{n}}\right\rfloor
=\left\{
\begin{array}
[c]{cc}%
1 & j\geq p^{n}-1+h\\
0 & j<p^{n}-1+h
\end{array}
\right.  .
\]
Since $w_{h}\left(  j\right)  \leq d_{h}\left(  j\right)  ,$ which is readily
seen by setting $i=0$ in the definition of $w_{h}\left(  j\right)  ,$ the
statement $w_{h}\left(  j\right)  =d_{h}\left(  j\right)  $ for all $0\leq
j\leq p^{n}-1$ is true if and only if $w_{h}\left(  j\right)  =1$ whenever
$j\geq p^{n}-1+h.$ Suppose $h>\left(  1-p^{n}\right)  /2$ and $d_{h}\left(
j\right)  =1.$ Then $j>p^{n}-1+\left(  1-p^{n}\right)  /2=\left(
p^{n}-1\right)  /2.$ Now assume there exists an $i$ such that $d_{h}\left(
i+j\right)  -d_{h}\left(  i\right)  =0$ and $i_{s}+j_{s}\leq p-1$ for all $s$.
Then $d_{h}\left(  i+j\right)  \geq d_{h}\left(  j\right)  =1$ so
$d_{h}\left(  i\right)  =1$ as well. Thus $i>\left(  p^{n}-1\right)  /2$. But
then $i+j\geq p^{n},$ contradicting the fact that $i_{s}+j_{s}\leq p-1$ for
all $s$. Therefore, no such $i$ can occur, hence $w_{h}\left(  j\right)
=d_{h}\left(  j\right)  $ for all $j$ and $\mathfrak{P}_{L}^{h}=\mathfrak{A}%
_{h}\cdot\rho.$

Now suppose that $h\leq\left(  1-p^{n}\right)  /2$ and let $j=p^{n}+h-1$. Then
$d_{h}\left(  j\right)  =1.$ Let
\[
i=p^{n}-1-j=p^{n}-1-\left(  p^{n}+h-1\right)  =-h.
\]
Then $i_{s}+j_{s}=p-1$ for all $s$. As above, $d_{h}\left(  i+j\right)  =1.$
But $i=-h<p^{n}-1-h$ so $d_{h}\left(  i\right)  =0.$ Thus $w_{h}\left(
j\right)  =w_{h}\left(  p^{n}+h-1\right)  =0$ and $\mathfrak{P}_{L}^{h}$ is
not free over $\mathfrak{A}_{h}.$

We summarize, generalizing to all $h\in\mathbb{Z}.$

\begin{theorem}
Let $H=H_{n,r,f}^{\ast},$ $0<r<n\leq2r,\;f\in K^{\times}.$ Suppose
$v_{L}\left(  x\right)  =-1$ and $v_{L}\left(  \rho\right)  =1.$ Let
$h\in\mathbb{Z},$ and let $m=\left\lfloor h/p^{n}\right\rfloor .$ Then
$\mathfrak{P}_{L}^{h}$ is free over $\mathfrak{A}_{h}$ if and only if
$\operatorname*{res}\left(  h-2\right)  >\left(  p^{n}-3\right)  /2;$ under
this restriction, $\mathfrak{P}_{L}^{h}=\mathfrak{A}_{h}\cdot\left(  T^{m}%
\rho\right)  .$
\end{theorem}

\begin{remark}
Notice that we do not need $v_{K}\left(  f\right)  \geq2\left(  1-p^{-n}%
\right)  +p^{r+1-n}$ in the statement above since, for any $f\in K^{\times},$
an $H_{n,r,f}^{\ast}$ of suitably high tolerance exists.
\end{remark}

\begin{proof}
Consider first the case $2-p^{n}\leq h\leq1.$ Then, $0\leq h-2+p^{n}\leq
p^{n}-1.$ We have seen that $\mathfrak{P}_{L}^{h}$ is $\mathfrak{A}_{h}$-free
if and only if $h>\left(  1-p^{n}\right)  /2,$ and since $h\leq1$ this
inequality holds if and only if
\[
\frac{p^{n}-3}{2}<h-2+p^{n}\leq p^{n}-1.
\]
Thus, $\mathfrak{P}_{L}^{h}$ is $\mathfrak{A}_{h}$-free if and only if
$\operatorname*{res}\left(  h-2\right)  >\left(  p^{n}-3\right)  /2.$

Now for more general $h$, $\mathfrak{P}_{L}^{h}$ is free over $\mathfrak{A}%
_{h}$ if and only if $\mathfrak{P}_{L}^{\operatorname*{res}\left(  h\right)
}$ is free over $\mathfrak{A}_{\operatorname*{res}\left(  h\right)
}=\mathfrak{A}_{h}$, so we have freeness if and only if%
\[
\operatorname*{res}\left(  \operatorname*{res}\left(  h\right)  -2\right)
>\left(  p^{n}-3\right)  /2,
\]
and since the left-hand side reduces to $\operatorname*{res}\left(
h-2\right)  $ we get the inequality desired. That $\mathfrak{P}_{L}%
^{h}=\mathfrak{A}_{h}\cdot\left(  T^{m}\rho\right)  $ is immediate since
$\mathfrak{P}_{L}^{h}=T^{m}\mathfrak{P}_{L}^{\operatorname*{res}\left(
h\right)  }.$
\end{proof}

In particular, notice that $\mathfrak{O}_{L}$ is free over $\mathfrak{A}$ when $b=1.$

\section{Picking the Best Hopf Algebra and Action}

In the examples provided here -- with $L=K\left(  x\right)  ,\;v_{L}\left(
x\right)  =b,\;p\nmid b$ -- questions concerning the Hopf Galois module
structure of $\mathfrak{O}_{L}$ have little to do with the exact Hopf algebra
chosen. For any choice of $0<r<n\leq2r$ and $v_{K}\left(  f\right)
\geq2-p^{n}\left(  1-b\left(  p^{r+1}-1\right)  \right)  $ we have scaffolds
of sufficiently high tolerance, and their existence allows us to apply the
numerical data of Proposition \ref{bigbce}. So, if $\mathfrak{P}_{L}^{h}$ is
free over $\mathfrak{A}_{h}$ for $H=H_{n,r,f}^{\ast},$ then $\mathfrak{P}%
_{L}^{h}$ is free over $\mathfrak{A}_{h}$ for any $H=H_{n,r^{\prime}%
,f^{\prime}}^{\ast},\;0<r^{\prime}<n\leq2r$ and $v_{K}\left(  f^{\prime
}\right)  \geq2-p^{n}\left(  1-b\left(  p^{r+1}-1\right)  \right)  $.
Additionally, the description of $\mathfrak{A}_{h}$ given in Proposition
\ref{bigbce} is independent of which Hopf algebra $H$ is chosen since the
value of $T^{-w_{h}\left(  j\right)  }$ is independent of $H$; of course, the
actual elements $z_{p^{s}}$ depend on the chosen $H.$

In addition to the family constructed here, the divided power $K$-Hopf
algebra$\;A\;$of rank $p^{n}$ found in (\cite[Ex. 5.6.8]{Montgomery93}, where
it is denoted $H$) acts on $L$: in terms of its dual, $A^{\ast}$ represents
the $n^{\text{th}}$ Frobenius kernel of the additive group scheme, and its
simple coaction is given by Chase in \cite{Chase76}. In \cite[Sec.
5.2]{ByottChildsElder14} a scaffold of infinite tolerance (so the congruences
are replaced by equalities) is constructed for $A$. Their scaffold is similar
to our constructions -- indeed, for large values of $v_{K}\left(  f\right)  ,$
$A^{\ast}$ and $H_{n,r,f}$ act very similarly on $L,$ and we can view
$H_{n,r,f}$ as a deformation of $A^{\ast}.$

Thus, it is natural to ask: which Hopf algebra is ``best''? As the
determination of integral Hopf Galois module structure does not depend on the
choice of $H$, there would need to be further properties of interest to make a distinction.

For a single choice of $H_{n,r,f},$ different actions lead to scaffolds of
different tolerances, though we can always make $\mathfrak{T}$ arbitrarily
large. So here, we may ask: which action is the ``best''? If one is primarily
interested in describing $\mathfrak{O}_{L}$ as an $\mathfrak{A}$-module then
the action where $L=K\left(  x\right)  ,$ $v_{L}\left(  x\right)  =-1$ appears
to be a good choice since $\mathfrak{O}_{L}$ is free over $\mathfrak{A}$
whenever $v_{K}\left(  f\right)  \geq3.$ If, on the other hand, one is
primarily interested in describing $\mathfrak{P}_{L}^{h}$ for a specific value
of $h$ there may be better choices. For example, in an unpublished work by
Jelena Sundukova, she states that $\mathfrak{P}_{L}^{h}$ is a free
$\mathfrak{A}_{h}$-module if $v_{L}\left(  x\right)  =-h.$ Her work also
describes choices of $v_{L}\left(  x\right)  $ which make $\mathfrak{P}%
_{L}^{h}$ free over $\mathfrak{A}_{h}$ reasonably rare, for example
$v_{L}\left(  x\right)  =p^{n}-2.$ As with choosing the Hopf algebra, we would
need to have more properties of this action which we deem ``desirable'' in
order to pick one action over another.

\bibliographystyle{amsalpha}
\bibliography{MyRefs}
\end{document}